\documentclass{amsart}[12pt]
\def\doctype{}

\usepackage{latexsym,amssymb,amsfonts,dsfont,bm}
\usepackage{fancyhdr,caption}
\usepackage{tikz}
\usepackage{hyperref}

\newcommand\lam{\lambda}

\newcommand{\cB}{\mathcal{B}}

\newcommand\Z{\mathbb{Z}}
\newcommand\F{\mathbb{F}}

\newcommand{\comment}[1]{}

\numberwithin{equation}{section}

\fancypagestyle{titlepage}{
\fancyhead[R]{\doctype}
\fancyhead[CL]{}
\cfoot{\vspace{5pt} \thepage}
}

\let\oldsection\section
\newcommand\boldsection[1]{\oldsection{\bf #1}}
\newcommand\starsection[1]{\oldsection*{\bf #1}}
\makeatletter
\renewcommand\section{\@ifstar\starsection\boldsection}
\makeatother

\newtheoremstyle{theorem}
  {12pt}		  
  {0pt}  
  {\sl}  
  {\parindent}     
  {\bf}  
  {. }    
  { }    
  {}     
\theoremstyle{theorem}
\newtheorem{thm}{Theorem}[section]  
\newtheorem{cor}[thm]{Corollary}

\newtheorem{prop}[thm]{Proposition}

\newtheoremstyle{definition}
  {12pt}		  
  {0pt}  
  {}  
  {\parindent}     
  {\bf}  
  {. }    
  { }    
  {}     
\theoremstyle{definition}

\newtheorem{ex}[thm]{Example}

\newcommand\rk{{\sc Remark.} }

\renewcommand{\proofname}{Proof}

\makeatletter
\renewenvironment{proof}[1][\proofname]{\par
  \pushQED{\qed}%
  \normalfont \partopsep=\z@skip \topsep=\z@skip
  \trivlist
  \item[\hskip\labelsep
        \scshape
    #1\@addpunct{.}]\ignorespaces
}{%
  \popQED\endtrivlist\@endpefalse
}
\makeatother

\makeatletter
\renewcommand*\@maketitle{%
  \normalfont\normalsize
  \@adminfootnotes
  \@mkboth{\@nx\shortauthors}{\@nx\shorttitle}%
  \global\topskip42\p@\relax 
  \@settitle
  \ifx\@empty\authors \else {\vskip 1em
\vtop{\centering\shortauthors\@@par}} \fi
  \ifx\@empty\@date \else {\vskip 1em \vtop{\centering\@date\@@par}}\fi 
  \ifx\@empty\@dedicatory
  \else
    \baselineskip18\p@
    \vtop{\centering{\footnotesize\itshape\@dedicatory\@@par}%
      \global\dimen@i\prevdepth}\prevdepth\dimen@i
  \fi
  \@setabstract
  \normalsize
  \if@titlepage
    \newpage
  \else
    \dimen@34\p@ \advance\dimen@-\baselineskip
    \vskip\dimen@\relax
  \fi
} 
\renewcommand*\@adminfootnotes{%
  \let\@makefnmark\relax  \let\@thefnmark\relax
  \ifx\@empty\@subjclass\else \@footnotetext{\@setsubjclass}\fi
  \ifx\@empty\@keywords\else \@footnotetext{\@setkeywords}\fi
  \ifx\@empty\thankses\else \@footnotetext{%
    \def\par{\let\par\@par}\@setthanks}%
  \fi
\thispagestyle{titlepage}
}
\makeatother

\begin{document}

\title{\large Local Balance in Graph Decompositions}

\author{Flora C.~Bowditch and Peter J.~Dukes}
\address{
Mathematics and Statistics,
University of Victoria, Victoria, Canada
}
\email{bowditch@uvic.ca, dukes@uvic.ca}

\thanks{Research of the authors is supported by NSERC: for the first author by a CGS-M award, and for the second author by Discovery Grant 312595--2017}

\date{\today}

\begin{abstract}
In a balanced graph decomposition, every vertex of the host graph appears in the same number of blocks. 
We propose the use of colored loops as a framework for unifying various other types of local balance conditions in graph decompositions.  In the basic case where a single graph with colored loops is used as a block, an existence theory for such decompositions follows as a straightforward generalization of previous work on balanced graph decompositions. 
\end{abstract}

\maketitle
\hrule

\section{Introduction}
\label{intro}

\subsection{Overview}

Let $G$ be a finite undirected simple graph. A $G$-\emph{decomposition} of a complete graph $K_v$ is 
a collection of subgraphs $G_1,G_2,\dots,G_b$ of $K_v$ with each $G_i \cong G$ and such that the edge sets $E(G_i)$ partition $E(K_v)$.
This extends to decompositions of the $\lambda$-fold complete graph $K_v^\lam$, where every edge occurs exactly $\lam$ times; in this more general setting,
every pair of distinct vertices is an edge of some $G_i$ exactly $\lam$ times.  When $G$ is itself a clique $K_k$, this reduces to a $(v,k,\lam)$-BIBD. 
For this reason,  the term `$G$-design' is sometimes used, and it is natural to use the term `block' for subgraphs $G_i$ in the decomposition.

Suppose $G$ has $n$ vertices and $m>0$ edges.  Let $g$ be the gcd of vertex degrees in $G$.
By counting in two ways the edges of $K_v^\lam$, it follows that existence of a $G$-decomposition of $K_v^\lam$ implies 
\begin{eqnarray} 
\label{G-global} 
2m & \mid & \lam v(v-1) ~\text{and}\\ 
\label{G-local} 
g & \mid & \lam (v-1). 
\end{eqnarray} 
R.M. Wilson showed in \cite{Wilson75} that, in the case $\lam=1$, if $G$ is fixed and $v$ is large, these divisibility conditions are sufficient for existence of a $G$-decomposition of $K_v$.  Then, Lamken and Wilson \cite{LW} showed as part of a much more general theory that, for a fixed positive integer $\lam$ and $v>v_0(G,\lam)$, \eqref{G-global} and \eqref{G-local} are sufficient for existence of  a $G$-decomposition of $K_v^\lam$.  

A $G$-decomposition of $K_v^\lam$ is called  \emph{balanced} or \emph{equireplicate} if every vertex of the host graph appears in an equal number of copies of $G$. This condition was first introduced by Hell and Rosa in \cite{Hell and Rosa}. Dukes and Malloch showed \cite{Peter and Amanda} that for a fixed graph $G$, there is a balanced $G$-decomposition of $K_v^{\lambda}$ for all sufficiently large integers $v$ satisfying \eqref{G-global} and \eqref{G-local}, except for a change to the definition of $g$.  In fact, a somewhat more general result was shown in which vertices of $G$ have zero or more loops, and every vertex of $K_v^\lam$ appears as a looped vertex equally often.

Lamken and Wilson's theory \cite{LW} involves decompositions of edge-colored complete multigraphs into a given edge-colored simple graph $G$.  The use of edge colors allows for modeling simultaneous pairwise balance conditions in block designs.  To illustrate the power of this general theory, they applied their result to produce existence results for Whist tournaments, Steiner pentagon systems, uniform group divisible designs, and both resolvable and near-resolvable designs.  A somewhat more general version in which $G$ need only be `colorwise simple' appears in \cite{DMW}.

Taking inspiration from this, we introduce here the use of colored loops.  This framework is motivated by several applications which impose local balance conditions in graph decompositions.

To offer a na\"ive example, suppose in some experiment we wish to test all pairs of $v$ treatments equally often in blocks of size $k$, but that the treatments within each block are assigned a `seat'.  That is, each block $B$ is to be turned into a $k$-tuple $(x_1,\dots,x_k)$ such that $B=\{x_1,\dots,x_k\}$.  To balance experimental side-effects, it is desired that every treatment occur precisely $r/k$ times in each of the $k$ seats.  This is possible \cite{Stinson-comm} for any $(v,k,\lam)$-BIBD in which $k \mid r$ as a straightforward consequence of Hall's theorem.
In any case, we may model such a structure as follows.  First, let  $K_k^*$ denote a clique on $k$ vertices in which each vertex has a loop of a different color.  Now, we seek a decomposition of (the ordinary edges of) the multigraph $K_v^\lam$ into copies of $K_k^*$ such that every vertex of the host graph appears with a loop of each color equally often.

Later, we apply our general framework in a few different ways to consider degree-balanced and orbit-balanced decompositions, equitable block colorings, and block orderings.  The framework itself even encapsulates challenging topics such as uniform resolvability and designs with large block sizes, although our basic constructions to follow are insufficient for these challenging problems.

\subsection{Divisibility conditions}

We establish here the necessary divisibility conditions for $G$-designs, where $G$ has colored loops that are to occur equally often at each element in the design. Ignoring loops, since the ordinary edges of $K_v^{\lambda}$ still partition into copies of (the ordinary edges of) $G$, the `global' divisibility condition \eqref{G-global} from before is again necessary.   It remains to analyze the local arithmetic condition.

As before, say $G$ has $n$ vertices, $m$ edges, and now loops of $c$ different colors. For a vertex $u \in V(G)$, let $d_u$ denote the number of ordinary edges incident with $u$, and let $e_{u,i}$ denote the number of loops of color $i$ at $u$. Let $\ell_i=\sum_{u}e_{u,i}$ denote the total number of loops of color $i$ in $G$, $i = 1, \ldots, c$.

			To model the loop balancing, we also include loops on the vertices of $K_v^{\lambda}$.  For a tuple $\bm{\mu}=(\mu_1,\dots,\mu_c)$ of nonnegative integers, let $K_v^{[\boldsymbol{\mu}; \lambda]}$ denote the multigraph on $v$ vertices with $\lambda$ edges between every pair of vertices and $\mu_i$ loops of color $i$ at every vertex.  This is the host graph for our decompositions. Counting as in \cite{Peter and Amanda}, we require 			
			\begin{equation}\label{single graph loop equation}
			\mu_i = \frac{\lambda \ell_i(v-1)}{2m}
			\end{equation}			
			loops of color $i$ at each vertex in $K_v^{[\boldsymbol{\mu}; \lambda]}$ for each $i=1,\dots,c$. 
			Locally, at any vertex $x$ of $K_v^{[\boldsymbol{\mu}; \lambda]}$ the blocks containing $x$ induce a simultaneous partition of the edges and loops incident with $x$.
			That is, we need a simultaneous integral solution $\{s_u\}$ to			
			\begin{equation}\label{single graph local degree condition}
			\sum_{u \in V(G)} s_ud_u = \lambda (v-1) \hspace{0.5cm}\text{and }
			\end{equation}			
			\begin{equation}\label{single graph local loop condition}
			\sum_{u \in V(G)} s_ue_{u,i} = \mu_i \hspace{0.5cm}\text{for each } i=1,\dots,c.
			\end{equation}			
			Wilson uses in \cite{Wilson75} that existence of an integral solution to  \eqref{single graph local degree condition} is equivalent to		
			$g \mid \lambda (v-1)$, where $g=\gcd\{d_u: u\in V(G)\}$.

			Putting together \eqref{G-global} and (\ref{single graph loop equation}-\ref{single graph local loop condition}) we obtain our necessary divisibility conditions		
			\begin{equation}\label{single graph necessary conditions}
			\begin{aligned}
			2m & \mid \lambda v(v-1) \\		
			\alpha & \mid \lambda (v-1),
			\end{aligned}		
			\end{equation}			
			where $\alpha$ is the least positive integer such that			
			\begin{equation}\label{single graph alpha definition}
\alpha \left( 1, \frac{\ell_1}{2m} , \dots, \frac{\ell_c}{2m} \right) \in \sum_{u \in V(G)} \left( d_u ,e_{u,1}, \dots, e_{u,c} \right) \Z.
			\end{equation}
			The right side of \eqref{single graph alpha definition} is an integer lattice in dimension $c+1$ spanned by `degree-loop' vectors.
			For a given graph $G$, we say that the integers $\lambda$ and $v$ are \textit{admissible} for $G$ if they satisfy (\ref{single graph necessary conditions}). It is important to note that these necessary conditions will give different admissible values of $\lambda$ and $v$ depending on where loops are placed in $G$.	

\begin{ex}
Consider the two graphs shown in Figure~\ref{loop-placement}, both of which has three loops of each of two colors (red/vertical or blue/horizontal) on the same underlying graph.  The local condition demands that $\alpha(1,\frac{1}{3},\frac{1}{3})$ be an integral combination of degree-loop vectors.  For the graph on the left, we have $\alpha=6$, while for the graph on the right we have $\alpha=12$.  In the case $\lam=2$, congruence classes $v \equiv 4,10 \pmod{12}$ are admissible for the graph on the left but not the graph on the right.
\end{ex}

\begin{figure}[htbp]
\begin{tikzpicture}
\draw[thick,red] (1,1) to [out=135,in=180] (1,1.5);
\draw[thick,red] (1,1) to [out=45,in=0] (1,1.5);

\draw[thick,red] (-1,1) to [out=135,in=180] (-1,1.5);
\draw[thick,red] (-1,1) to [out=45,in=0] (-1,1.5);

\draw[thick,red] (0,0) to [out=135,in=180] (0,.5);
\draw[thick,red] (0,0) to [out=45,in=0] (0,.5);

\draw[thick,red] (1,-1) to [out=225,in=180] (1,-1.5);
\draw[thick,red] (1,-1) to [out=315,in=0] (1,-1.5);

\draw[thick,blue] (0,0) to [out=315,in=270] (0.5,0);
\draw[thick,blue] (0,0) to [out=45,in=90] (0.5,0);

\draw[thick,blue] (1,-1) to [out=315,in=270] (1.5,-1);
\draw[thick,blue] (1,-1) to [out=45,in=90] (1.5,-1);

\draw[thick,blue] (-1,1) to [out=135,in=90] (-1.5,1);
\draw[thick,blue] (-1,1) to [out=225,in=270] (-1.5,1);

\draw[thick,blue] (-1,-1) to [out=135,in=90] (-1.5,-1);
\draw[thick,blue] (-1,-1) to [out=225,in=270] (-1.5,-1);

\draw (0,0)--(1,1);
\draw (0,0)--(1,-1);
\draw (0,0)--(-1,1);
\draw (0,0)--(-1,-1);
\draw (1,1)--(1,-1);
\draw (-1,1)--(-1,-1);

\filldraw (0,0) circle [radius=.1];
\filldraw (1,1) circle [radius=.1];
\filldraw (1,-1) circle [radius=.1];
\filldraw (-1,1) circle [radius=.1];
\filldraw (-1,-1) circle [radius=.1];
\end{tikzpicture}
\hspace{2cm}
\begin{tikzpicture}
\draw[thick,red] (1,1) to [out=135,in=180] (1,1.5);
\draw[thick,red] (1,1) to [out=45,in=0] (1,1.5);

\draw[thick,red] (-1,1) to [out=135,in=180] (-1,1.5);
\draw[thick,red] (-1,1) to [out=45,in=0] (-1,1.5);

\draw[thick,red] (0,0) to [out=135,in=180] (0,.5);
\draw[thick,red] (0,0) to [out=45,in=0] (0,.5);

\draw[thick,red] (1,-1) to [out=225,in=180] (1,-1.5);
\draw[thick,red] (1,-1) to [out=315,in=0] (1,-1.5);

\draw[thick,blue] (1,1) to [out=315,in=270] (1.5,1);
\draw[thick,blue] (1,1) to [out=45,in=90] (1.5,1);

\draw[thick,blue] (1,-1) to [out=315,in=270] (1.5,-1);
\draw[thick,blue] (1,-1) to [out=45,in=90] (1.5,-1);

\draw[thick,blue] (-1,1) to [out=135,in=90] (-1.5,1);
\draw[thick,blue] (-1,1) to [out=225,in=270] (-1.5,1);

\draw[thick,blue] (-1,-1) to [out=135,in=90] (-1.5,-1);
\draw[thick,blue] (-1,-1) to [out=225,in=270] (-1.5,-1);

\draw (0,0)--(1,1);
\draw (0,0)--(1,-1);
\draw (0,0)--(-1,1);
\draw (0,0)--(-1,-1);
\draw (1,1)--(1,-1);
\draw (-1,1)--(-1,-1);

\filldraw (0,0) circle [radius=.1];
\filldraw (1,1) circle [radius=.1];
\filldraw (1,-1) circle [radius=.1];
\filldraw (-1,1) circle [radius=.1];
\filldraw (-1,-1) circle [radius=.1];
\end{tikzpicture}
\caption{Divisibility conditions depend on placement of loops} 
\label{loop-placement}
\end{figure}
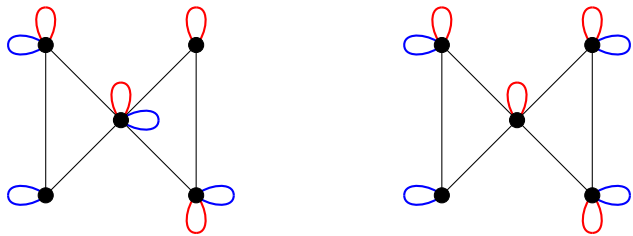

We are ready to state our main result.
					
			\begin{thm}\label{main}
				Let $\lambda$ be a positive integer. Suppose $G$ is a simple graph with $n$ vertices, $m>0$ edges, and $\ell_i$ loops of color $i$ for $i=1,\dots,c$. Then there exists a $G$-decomposition of $K_v^{[\boldsymbol{\mu}; \lambda]}$ for all sufficiently large integers $v$ satisfying the necessary conditions given in \eqref{single graph necessary conditions} and with multiplicities $\bm{\mu}$ given by \eqref{single graph loop equation}.
			\end{thm}	

\subsection{Outline}

In the next section, we survey some applications of Theorem~\ref{main} that showcase the utility of colored loops.  Section~\ref{sec:proof} contains the proof of the main theorem, which closely imitates the proof of the analogous result for loops of one color in \cite{Peter and Amanda}.  
For completeness, we review this proof method and highlight the key differences important for our setting.  In Section~\ref{sec:extensions}, we give some remarks on possible extensions of the problem, including to digraphs and families of allowed graphs.  A general existence result in the latter case is likely to require significant new ideas, or at least some simplifying hypotheses, as we illustrate with some examples.

\section{Applications}

\subsection{Degree-balanced and orbit-balanced decompositions}

			Bonisoli, Bonvicini and Rinaldi have introduced two slightly more restrictive types of balanced graph decompositions in \cite{hierarchy}: degree-balanced and orbit-balanced. For a given simple (loopless) graph $H$, let $D(H)$ be the set of all degrees of the vertices of $H$. For each $d \in D(H)$, the subset of all vertices of degree $d$ is the \textit{degree-class} defined by $d$. The degree classes partition $V(H)$. Given an $H$-decomposition of the complete graph $K_v$, let $r_d(u)$ denote the number of blocks containing $u$ as a vertex of degree $d$. An $H$-design is called \textit{degree-balanced} if for each $d \in D(H)$, $r_d(u)$ is independent of $u$.
					
			Let $A(H)$ be the set of vertex-orbits of $H$ under its automorphism group.  Given an $H$-design,
			let $r_{a}(u)$ denote the number of blocks containing $u$ as a vertex in orbit $a$, where $a \in A(H)$ and $u \in V(K_v)$.
			An $H$-decomposition of $K_v$ is called \textit{orbit-balanced} if for each $a \in A(H)$, $r_{a}(u)$ is independent of $u$. Since each orbit-class contains vertices of a common degree, it is clear that orbit-balanced graph decompositions are also degree-balanced. It is also easy to see that both degree- and orbit-balanced graph decompositions are balanced. However, the converse of each of these statements is not true in general, as seen by examples of Bonvicini in \cite{degree and orbit}.

As a direct consequence of our main result, we obtain a general existence result for each of these variants of balanced graph decompositions.

			\begin{cor}\label{degree-balanced corollary to main result}
				Let $\lambda \geq 0$. Suppose $H$ is a simple graph with $n$ vertices, $m$ edges, and degree set $D(H)$. Then there exists a degree-balanced $H$-decomposition of $K_v^{\lambda}$ for all sufficiently large $v$ satisfying \eqref{single graph necessary conditions} with $e_{u,d}=1$ if $\deg(u)=d$, and 0 otherwise.
			\end{cor}
			
			\begin{cor}\label{orbit-balanced corollary to main result}
				Let $\lambda \geq 0$. Suppose $H$ is a simple graph with $n$ vertices, $m$ edges, and vertex-orbit set $A(H)$. Then there exists an orbit-balanced $H$-decomposition of $K_v^{\lambda}$ for all sufficiently large $v$ satisfying \eqref{single graph necessary conditions} with $e_{u,a}=1$ if $u$ belongs to orbit $a$, and 0 otherwise.
			\end{cor}

It is straightforward to see that  \eqref{single graph necessary conditions} with the indicated loop multiplicities are necessary conditions for each of these types of designs.

\subsection{Equitable Block colorings}\label{Equitable Block colorings}
		
			Let $G=(V,E)$ be a graph, and let $\mathcal{F}$ be a $G$-decomposition of $K_v$. An $s$-\textit{equitable block-coloring} of $\mathcal{F}$ is a coloring $f: \mathcal{F} \rightarrow \{1,\ldots,s\}$ of the blocks such that for each vertex $u \in V$ and any two colors $i \neq j$, we have $$|b(f, u, i) - b(f, u, j)| \leq 1,$$
where $b(f, u, i)$ is the number of blocks in $\mathcal{F}$ containing $u$ that are colored $i$ by $f$. Informally, it is an assignment of $s$ colors to the blocks in $\mathcal{F}$ so that every vertex appears as equally as possible in blocks of each of the colors. In \cite{equitable colorings 1}, M. Gionfriddo and Quattrocchi investigated equitable colorings of 4-cycle systems.
			
Equitable block colorings are closely related to `resolvability' questions.  In particular, we say that a $(v,k,\lam)$-BIBD, say $(V,\cB)$, is \emph{resolvable} if its block collection $\cB$ can be resolved into partitions of $V$ also known as \emph{parallel classes}.  To model resolvable designs with loop colors requires $\lam(v-1)/(k-1)$ distinct loop colors, a function of $v$, and is presently outside the scope of Theorem~\ref{main}.  However, a relaxation studied in \cite{thickly resolvable} allows for parallel classes to be replaced by regular configurations of blocks.  In other words, we want an equitable block coloring of $\cB$ with $s$ colors used equally often at each element of $V$.  This can be modeled by a disjoint union of $s$ cliques $K_k$, each component of which has loops of a different color.  Theorem \ref{main} gives the following result.

\begin{cor}
Suppose $k\ge 2$, $s \ge 1$ and $\lam \ge 0$ are given integers.  There exists a $(v,k,\lam)$-BIBD having an $s$-equitable block coloring for all sufficiently large integers $v$ satisfying $2m \mid \lam v(v-1)$ and $s(k-1) \mid \lam(v-1)$.
\end{cor}

			We can also obtain a similar result for equitable block-colorings of $G$-designs. For a given graph $G$ and positive integer $s$, let $\mathcal{G}=\{G_1, \ldots, G_s\}$, where $G_i$ is the graph $G$ with a loop of color $i$ at every vertex. Let $G_0$ be an edge-disjoint union of the graphs in $\mathcal{G}$. Then by Theorem \ref{main}, we can obtain a $G_0$-decomposition of $K_v^{[\boldsymbol{\mu};1]}$ for all sufficiently large integers $v$ satisfying the necessary conditions in (\ref{single graph necessary conditions}). By definition, every element encounters exactly $\mu_i$ loops of color $i$, for $i = 1,\dots, c$. Hence, each vertex of $K_v^{[\boldsymbol{\mu};1]}$ must appear an equal number of times in the copies of $G_0$ as a vertex with a loop of color $i$. Removing the loops, we see that the result gives a $G$-decomposition of $K_v$ equipped with an $s$-equitable block-coloring where every vertex appears equally often in blocks of each color.  

			In \cite{equitable colorings 2}, L. Gionfriddo, M. Gionfriddo and Ragusa introduced a generalization of these colorings; their work was later extended by Li and Rodger \cite{equitable colorings 3}. An $(s, p)$-\textit{equitable block-coloring} of $\mathcal{F}$ is a coloring $f: \mathcal{F} \rightarrow \{1,2,\ldots,s\}$ such that
			\begin{itemize}
				\item for each $u \in V$, the blocks containing $u$ are colored using exactly $p$ colors, and
				\item for each $u \in V$ and for each $\{i, j\} \subset C(f, u)$,  $|b(f, u, i) - b(f, u, j)| \leq 1$,
			\end{itemize}
			where $C(f, u)$ is the set of colors used on blocks incident with vertex $u$ and $b(f, u, i)$ is as defined before.
In other words, an $(s, p)$-equitable block-coloring is an assignment of $s$ colors to the blocks in $\mathcal{F}$ so that each element is incident with blocks colored with exactly $p$ colors; and every vertex appears equally often (or as equally as possible) in blocks of each of the $p$ colors. Notice that when $p=s$, the definition of $s$-equitable block-coloring is recovered.
	
			To model these more general block-colorings, we may use the copies $\{G_1,\dots,G_s\}$ of $G$ with colored loops mentioned earlier, but this time a disjoint union of all $s$ will not work for arbitrary choices of $p$.  Instead, a family of $\binom{s}{p}$ disjoint unions with $p$ different colors is appropriate.  Some remarks on decompositions into families of allowed graphs are given in Section~\ref{sec:families}, although we are presently lacking any general result suitable for the problem.

\subsection{Block orderings}

Taking inspiration from Gray codes, it is of interest to order the blocks of a design so that consecutive blocks intersect.  The \emph{block intersection graph} of a design $(V,\cB)$ has as its vertex set $\cB$, and two blocks $B,B' \in \cB$ are declared adjacent if $B \cap B' \neq \emptyset$.  It was shown in \cite{HR} that the block intersection graph of any $(v,k,1)$-BIBD is Hamiltonian, thus settling in a strong sense the block ordering problem discussed above.

In a little more generality, we consider a graph $G$ with two distinguished vertices $s$ and $t$.  We ask when a $G$-decomposition of $K_v$ admits a (cyclic) ordering of its  $G$-blocks so that consecutive blocks $H,H'$ share the vertex $f_H(t)=f_{H'}(s)$, where $f_H$ denotes the natural embedding of $V(G)$ into $V(K_v)$ defining the block $H$.  As an example, suppose $G=P_4$, a path on 4 vertices, with its endpoints taking the role of $s,t$.  A $G$-decomposition of $K_v$ with the ordering described above is equivalent to an Eulerian trail in $K_v$ which can be cut into $\frac{1}{3} \binom{v}{2}$ consecutive copies of $P_4$.  We note that there exist $P_4$-decompositions of $K_v$ which admit no such ordering.

As an application of our main theorem, we have the following consequence for $(s,t)$-cyclic block orderings.

\begin{prop}
Given a graph $G$ and distinct vertices $s,t \in V(G)$, there exists, for all
sufficiently large integers $v \equiv 1 \pmod{2m}$, a $G$-decomposition of $K_v$ with an $(s,t)$-cyclic block ordering.
\end{prop}

\begin{proof}
We consider the graph $G^*$ with a red loop at $s$ and a blue loop at $t$.  By Theorem~\ref{main}, there exists, for sufficiently large $v \equiv 1 \pmod{2m}$, a $G^*$-decomposition of $K_v^{[\boldsymbol{\mu}; \lambda]}$, equipped with $(v-1)/2m$ loops of each color at each vertex.  Given such a decomposition, we order its blocks as follows.  Begin with a block $H_1$ and consider its vertex $x_1$ with a blue loop.  In some other block, say $H_2$, vertex $x_1$ appears with a red loop.  We then use the blue loop in $H_2$ at, say, $x_2$ and continue sequencing blocks in this way such that a longest such chain $\mathcal{C}$ of blocks has been created.  If not all blocks belong to $\mathcal{C}$, then some vertex $x$ of the host graph has fewer than $(v-1)/2m$ blue loops in blocks of $\mathcal{C}$.  But $\mathcal{C}$ contributes equally many loops of each color at every vertex, so we may continue sequencing until we return to a red loop at vertex $x$.  This contradicts that $\mathcal{C}$ was longest, and implies that all $v(v-1)/2m$ blocks were sequenced in $\mathcal{C}$.
\end{proof}

\section{Proof of the main result}
\label{sec:proof}

We are ready to prove Theorem \ref{main}. The proof will be carried out in several steps, following closely the approach in \cite{Peter and Amanda,Wilson75}. First, we obtain constructions for $v=q$, a large prime power congruent to $1 \pmod{2m}$, and any $\lam$.  By Dirichlet's theorem, this provides infinitely many $G$-designs of different orders.
We then obtain `signed' $G$-designs under a very mild assumption $v \ge n+2$.  This leads to $G$-designs with large $\lam' \gg \lam$.  We use an algebraic construction of Wilson \cite{RMW2} to `stretch' such a $G$-design into one with the desired $\lam$ on a  larger number of vertices.  Finally, PBD closure is invoked to get eventual periodicity of the integers $v$ which permit one of the preceding constructions.
	
		\subsection{Cyclotomic construction for large prime powers}\label{Construction for Large Prime Powers}
	
The following construction appears as \cite[Proposition 1]{Wilson75} for graphs without loops, and \cite[Lemma 2.1]{Peter and Amanda} for graphs with loops of one color.  It works the same  when loop colors are introduced.

			\begin{prop}\label{construction for large prime powers}
				Let $G$ be an undirected graph with $n$ vertices, $m$ edges, and $\ell_i$ loops of color $i$ for $i=1,\dots,c$. Then $K_q^{[\boldsymbol{\mu}; \lambda]}$ can be decomposed into copies of $G$ for all prime powers $q$ satisfying $q \equiv 1 \pmod{2m}$ and $q > m^{n^2}$. 
			\end{prop}	
		
			\begin{proof}
				Observe that $q$, taking the role of $v$, satisfies (\ref{single graph necessary conditions}) since $\alpha \mid 2m$. Let $G'$ denote $G$ with loops removed. Existence of a $G'$-decomposition of $K_q^{\lambda}$ follows from the following cyclotomic method in \cite{Wilson75}. A base block $G_0 \cong G'$ with $V(G_0) = \mathbb{F}_q$ is found by distributing the vertices of $G'$ so that the $m$ edge differences lie in distinct cosets of a subgroup $C_0\subset \mathbb{F}_q^\times$ of index $m$. The construction then develops this base block as $\lam$ copies of the family
\begin{equation}
\label{cyclotomic-family}
\{t G_0 + a: t \in T, a \in \F_q\},
\end{equation}
where $T$ is a transversal of $\{1,-1\}$ in $C_0$ and where arithmetic on $G_0$ denotes copies of $G'$ under the corresponding vertex permutations.
				Now, apply the same construction to $G$-blocks (with loops). Since the family \eqref{cyclotomic-family} is closed under additive shift in $\mathbb{F}_q$, it follows that every vertex $u$ of $G$ appears in a block at each element of $\mathbb{F}_q$ equally often, namely $\frac{\lambda (q-1)}{2m}$ times. Summing over $u$, every element of $\mathbb{F}_q$ accumulates a total of exactly $\frac{\lambda l_i(q-1)}{2m}=\mu_i$ loops of color $i$.
			\end{proof}
	
\rk
The guarantee $q>m^{n^2}$ in Proposition~\ref{construction for large prime powers} is often in practice much larger than necessary.

\begin{ex}
Let $G'$ denote the path on 4 vertices.  In Figure~\ref{cyclotomy}, we illustrate a base block for a $G$-decomposition of $K_q$ when $q=7$. If $G$ includes loops placed on $G'$, the loops will distribute evenly when the block is developed additively in $\F_q$.
\end{ex}

\begin{figure}[htbp]
\begin{tikzpicture}
\foreach \a in {0,1,...,6}
\draw (90:2)--(90+4*360/7:2);
\draw (90+4*360/7:2)--(90+6*360/7:2);
\draw (90+6*360/7:2)--(90+5*360/7:2);

\draw[thick,blue] (0,2) to [out=135,in=90] (-0.5,2);
\draw[thick,blue] (0,2) to [out=225,in=270] (-0.5,2);

\draw[thick,blue] (0,2) to [out=135,in=90] (-0.5,2);
\draw[thick,blue] (0,2) to [out=225,in=270] (-0.5,2);

\draw[thick,blue] (1.56,1.25) to [out=315,in=270] (2.06,1.25);
\draw[thick,blue] (1.56,1.25) to [out=45,in=90] (2.06,1.25);

\draw[thick,red] (1.56,1.25) to [out=45,in=0] (1.56,1.75);
\draw[thick,red] (1.56,1.25) to [out=135,in=180] (1.56,1.75);

\foreach \a in {0,1,...,6}
\filldraw (90+\a*360/7:2) circle [radius=.1];
\end{tikzpicture}
\caption{Base block for a cyclic (loop-balanced) decomposition} 
\label{cyclotomy}
\end{figure}
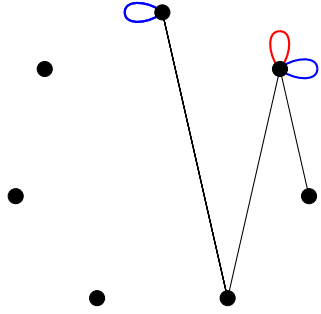

		\subsection{Integral Solutions}\label{Integral Solutions}
			
			The divisibility conditions \eqref{single graph necessary conditions} are not in general sufficient. However, they are enough for the existence of a `signed' $G$-decomposition, where negative copies of $G$ are allowed.
			
			\begin{prop}\label{integral solutions}
				Let $G$ be a graph with $n$ vertices, $m$ edges, and $\ell_i$ loops of color $i$ for $i=1,\dots,c$. Let $\mathcal{D}_v$ be the set of subraphs of $K_v$ which are isomorphic to $G$ without loops. If $v\geq n+2$, and $v,\lam,\bm{\mu}$ satisfy the congruences in \eqref{single graph necessary conditions}, then there exist integers $x_{H}$ for each $H \in \mathcal{D}_v$ such that
				\begin{equation}\label{integral solutions equation 1} 
				\sum\limits_{H: st \in E(H)} x_H = \lambda
				\end{equation} for every edge $\{s,t\} \in E(K_v)$ and
				\begin{equation}\label{integral solutions equation 2} 
				\sum\limits_{H: u \in V(H)} e_{u,i} x_H = \mu_i
				\end{equation}
				for every vertex $u \in V(K_v)$ and every color $i=1,\dots,c$.
			\end{prop}

\begin{proof}
We follow the same strategy used in \cite[Lemma 2.3]{Peter and Amanda} and \cite[Proposition 4]{Wilson75}.  It suffices to show that for any assignment of integers $\beta_{st}$ to the edges $st \in E(K_v)$, and $\beta_u^i$, $i=1,\dots,c$,  to the vertices $u \in V(K_v)$ such that for each subgraph $H$ the sum 
				$$\sigma_H = \sum\limits_{st\in E(H)} \beta_{st} + \sum\limits_{i=1}^c\sum\limits_{u \in V(H)} e_{u,i} \beta_u^i$$ is divisible by some integer $d$, then the sum
				$$\sigma = \lambda\sum\limits_{st\in E(K_v)} \beta_{st} + \sum\limits_{i=1}^c\mu_i\sum\limits_{u \in V(K_v)} \beta_u^i$$ is also divisible by $d$.
Applying a vertex swap to one copy of $H$ on $K_v$, one obtains, as in \cite{Peter and Amanda,Wilson75}, that 
\begin{equation}
\label{perm2}
\sum_{x \in N_H(s)} \beta_{sx} + \sum\limits_{i=1}^c e_{s,i} \beta_s^i \equiv 
\sum_{x \in N_H(s)} \beta_{tx} + \sum\limits_{i=1}^c e_{t,i} \beta_t^i \pmod{d}.
\end{equation}
Likewise, by applying a pair of disjoint vertex swaps, we obtain that $\beta_{st}$ is a coboundary (mod $d$); that is, there exist integers $\epsilon$ and $b_s$, $s \in V(K_v)$, so that
\begin{equation}
\label{perm4}
\beta_{st} \equiv b_s+b_t+\epsilon \pmod{d}.
\end{equation}
Using \eqref{perm2} and \eqref{perm4}, we can write
\begin{equation}
\label{sigma-H}
\sigma_H \equiv 2m b_0 + \sum_i l_i \beta_0^i+m\epsilon \pmod{d}
\end{equation}
and, computing as in \cite[eq (14)]{Peter and Amanda},
\begin{align*}
\sigma & \equiv \lam (v-1)  \left(  v b_0 + \sum_i \frac{ l_i v}{2m} \beta_0^i + \frac{v}{2} \epsilon \right)  \\
&\equiv \frac{\lam v(v-1)}{2m}(2m b_0 + \sum_i l_i \beta_0^i+m\epsilon) \\
&\equiv \frac{\lam v(v-1)}{2m} \sigma_H \pmod{d}.
\end{align*}
The result now follows.
\end{proof}

\rk
The hypothesis $v \ge n+2$ in Proposition~\ref{integral solutions} can in many cases be weakened to $v \ge n$.

\begin{ex}
Consider the graph $G$ on the left in Figure~\ref{signed-decomp}, with ordinary edges as $C_5$ and colored loop  multiplicites as indicated.  A signed combination of copies of $G$ decomposes $K_5^{[2,2;1]}$.
\end{ex}

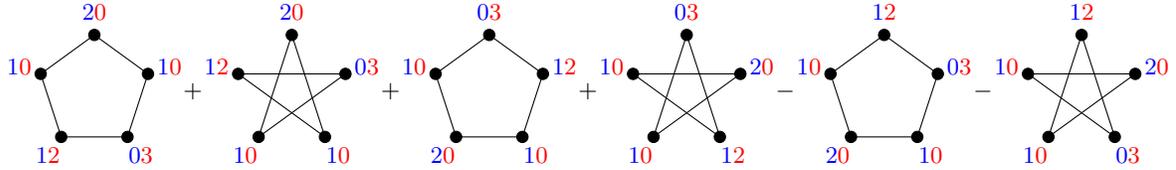
\begin{figure}[htbp]
\begin{tikzpicture}[scale=0.75]
\small
  \begin{scope}[shift={(0,0)}]
\foreach \a in {0,1,...,4}
\draw (90+\a*72:1)--(162+\a*72:1);
\foreach \a in {0,1,...,4}
\filldraw (90+\a*72:1) circle [radius=.1];
\node at (90:1.4) {$\color{blue}2\color{red}0$};
\node at (90+72:1.4) {$\color{blue}1\color{red}0$};
\node at (90+2*72:1.4) {$\color{blue}1\color{red}2$};
\node at (90+3*72:1.4) {$\color{blue}0\color{red}3$};
\node at (90+4*72:1.4) {$\color{blue}1\color{red}0$};
  \end{scope}

  \begin{scope}[shift={(3.5,0)}]
\foreach \a in {0,1,...,4}
\draw (18+\a*72:1)--(162+\a*72:1);
\foreach \a in {0,1,...,4}
\filldraw (90+\a*72:1) circle [radius=.1];
\node at (-1.75,0) {$+$};
\node at (90:1.4)  {$\color{blue}2\color{red}0$};
\node at (90+72:1.4)  {$\color{blue}1\color{red}2$};
\node at (90+2*72:1.4)  {$\color{blue}1\color{red}0$};
\node at (90+3*72:1.4)  {$\color{blue}1\color{red}0$};
\node at (90+4*72:1.4)  {$\color{blue}0\color{red}3$};
  \end{scope}

  \begin{scope}[shift={(7,0)}]
\foreach \a in {0,1,...,4}
\draw (90+\a*72:1)--(162+\a*72:1);
\foreach \a in {0,1,...,4}
\filldraw (90+\a*72:1) circle [radius=.1];
\node at (-1.75,0) {$+$};
\node at (90:1.4)  {$\color{blue}0\color{red}3$};
\node at (90+72:1.4) {$\color{blue}1\color{red}0$};
\node at (90+2*72:1.4)  {$\color{blue}2\color{red}0$};
\node at (90+3*72:1.4) {$\color{blue}1\color{red}0$};
\node at (90+4*72:1.4)  {$\color{blue}1\color{red}2$};
  \end{scope}

  \begin{scope}[shift={(10.5,0)}]
\foreach \a in {0,1,...,4}
\draw (18+\a*72:1)--(162+\a*72:1);
\foreach \a in {0,1,...,4}
\filldraw (90+\a*72:1) circle [radius=.1];
\node at (-1.75,0) {$+$};
\node at (90:1.4) {$\color{blue}0\color{red}3$};
\node at (90+72:1.4) {$\color{blue}1\color{red}0$};
\node at (90+2*72:1.4) {$\color{blue}1\color{red}0$};
\node at (90+3*72:1.4) {$\color{blue}1\color{red}2$};
\node at (90+4*72:1.4) {$\color{blue}2\color{red}0$};
  \end{scope}

  \begin{scope}[shift={(14,0)}]
\foreach \a in {0,1,...,4}
\draw (90+\a*72:1)--(162+\a*72:1);
\foreach \a in {0,1,...,4}
\filldraw (90+\a*72:1) circle [radius=.1];
\node at (-1.75,0) {$-$};
\node at (90:1.4) {$\color{blue}1\color{red}2$};
\node at (90+72:1.4) {$\color{blue}1\color{red}0$};
\node at (90+2*72:1.4) {$\color{blue}2\color{red}0$};
\node at (90+3*72:1.4) {$\color{blue}1\color{red}0$};
\node at (90+4*72:1.4) {$\color{blue}0\color{red}3$};
  \end{scope}

  \begin{scope}[shift={(17.5,0)}]
\foreach \a in {0,1,...,4}
\draw (18+\a*72:1)--(162+\a*72:1);
\foreach \a in {0,1,...,4}
\filldraw (90+\a*72:1) circle [radius=.1];
\node at (-1.75,0) {$-$};
\node at (90:1.4) {$\color{blue}1\color{red}2$};
\node at (90+72:1.4) {$\color{blue}1\color{red}0$};
\node at (90+2*72:1.4) {$\color{blue}1\color{red}0$};
\node at (90+3*72:1.4) {$\color{blue}0\color{red}3$};
\node at (90+4*72:1.4) {$\color{blue}2\color{red}0$};
  \end{scope}
\end{tikzpicture}
\caption{A loop-balanced signed decomposition} 
\label{signed-decomp}
\end{figure}

		\subsection{Wilson's Construction}\label{Wilson's Construction}
		

We show here how to obtain a $G$-design of order $v$ for each admissible residue class modulo $2m$.
			The solutions found in Section \ref{Integral Solutions} are allowed to use `negative copies' of $G$. 
			To obtain a genuine $G$-decomposition, we can uniformly raise the multiplicity of each copy of $G$ in $\mathcal{D}_v$ so as to overcome the negative multiplicity of every block. This  leaves us with a $G$-decomposition of $K_v^{\lam'}$, where possibly $\lambda' \gg \lambda$. Wilson's construction \cite{RMW2,Wilson75} stretches this $G$-design to one with the desired $\lambda$ on $v'$ vertices, where $v' > v$ and $v'$ hits the same residue class as $v$.
			\begin{prop}\label{wilson's construction}
				Let $G$ be a graph with $n$ vertices, $m$ edges, and $\ell_i$ loops of color $i$ for $i=1,\dots,c$.  Suppose $m \mid M$.  For every integer $v \ge n+2$ satisfying \eqref{single graph necessary conditions}, there exists an integer $v'\equiv v \pmod{2M}$ such that $K_{v'}^{[\boldsymbol{\mu}; \lambda]}$ can be $G$-decomposed.
			\end{prop}			
			
			\begin{proof}
				Let $\{x_{H} : H \in \mathcal{D}_v\}$ be an integral solution found from Proposition \ref{integral solutions}.  For some integer $t$, let $x_H'=x_H + t$ for every $H\in \mathcal{D}_v$. Then from (\ref{integral solutions equation 1}), we have
				
				\vspace{-0.3cm}
				
				$$\sum\limits_{H: st \in E(H)} x_H' = \lambda + t\lambda_0=\lambda q,$$ 
where $$\lambda_0=\frac{2m|\mathcal{D}_v|}{v(v-1)}~\text{and}~q=1+t\frac{\lambda_0}{\lambda}.$$ This gives us a multiset $\mathcal{H}$ of $G$-blocks in $\mathcal{D}_v$ such that each edge $\{s,t\} \in E(K_v)$ appears in exactly $\lam q$ blocks of $\mathcal{H}$.
				
				As in \cite{Peter and Amanda,Wilson75}, we may choose $t$ (and hence $q$) such that the following three conditions are satisfied:
				\begin{itemize}
					\item $x_H'>0$ for every $H\in \mathcal{D}_v$,
					\item $\lambda \mid t$, \text{ and}
					\item $q=1+t \lambda_0/\lam$ is a prime congruent to 1 modulo $2M$.
				\end{itemize}
				Note that $$\lambda'=q\lambda=\lambda + t\lambda_0\equiv \lambda \pmod{\lambda_0}$$ is the multiplicity of every edge in the resulting $G$-design.  Moreover, the number of loops of color $i$ is
				$$\mu_i' = \frac{q\lambda \ell_i(v-1)}{2m}=\frac{\lambda'\ell_i(v-1)}{2m}$$
				for $i=1,\dots,c$.
				We now follow the same algebraic construction as developed by Wilson for block designs.  The extension to $G$-designs appears in \cite{Peter and Amanda,Wilson75}, which can be consulted for additional details.  We give an outline below for completeness, checking the colored loop condition.
				
				First, we choose $t \geq v^2$ large enough that Proposition \ref{construction for large prime powers} applies with $q^t$ taking the role of $q$.  Let $\Gamma$ denote the complete multipartite graph with $v$ parts, each of size $q^t$. Observe that $\Gamma$ has $v'=vq^t \equiv v \pmod{2M}$ vertices. The construction then produces a set of $G$-blocks which decompose $\Gamma$, and which are invariant under additive shifts in $\mathbb{F}_{q^t}$ on each part. Just as we observed in the proof of Proposition \ref{construction for large prime powers}, the additive automorphism guarantees that every vertex of $\Gamma$ encounters the same number of loops of a given color. In particular, every vertex sees $\lambda \ell_i q^t (v-1)/2m$ loops of color $i$ for each $i$.  This was also observed in \cite{Peter and Amanda} for the single-color case. Now we apply Proposition \ref{construction for large prime powers} and include blocks to decompose $K_{q^t}^{[\boldsymbol{\mu}(q^t); \lambda]}$ on each partite set of $\Gamma$. Here, 
				$$\boldsymbol{\mu}(q^t) = \lam (q^t-1) \Big(\frac{\ell_1}{2m}, \ldots, \frac{\ell_c}{2m} \Big).$$ This results in a $G$-decomposition of $K_{v'}^{[\boldsymbol{\mu}; \lambda]}$. Adding the loop multiplicites together, we have $$\mu_i =\frac{\lambda \ell_i q^t (v-1)}{2m} + \frac{\lambda \ell_i (q^t-1)}{2m} = \frac{\lambda \ell_i (v'-1)}{2m}$$ for each color $i$, 
as desired.			\end{proof}

\subsection{PBD Closure}\label{PBD Closure}
		
			Our final step is to use pairwise balanced designs to complete each residue class that satisfies the divisibility conditions \eqref{single graph necessary conditions}. This also works similarly as \cite{Peter and Amanda,Wilson75}.
			
Let $v$ be a positive integer and $K \subseteq \Z_{\ge 2} :=
\{2,3,4,\dots\}$.  A \emph{pairwise balanced design}
PBD$(v,K)$ is a pair $(V,\cB)$, where
\begin{itemize}
\item
$V$ is a $v$-element set of \emph{points};
\item
$\cB \subseteq \cup_{k \in K} \binom{V}{k}$ is a family of subsets of
$V$, called \emph{blocks}; and
\item
any two distinct points appear together in exactly one block.
\end{itemize}
In alternative language, a PBD$(v,K)$ is an edge-decomposition of the complete graph of order $v$ into
cliques whose sizes come from the set $K$.  A PBD$(v,K)$ with $K=\{k\}$ is a $(v,k,1)$-BIBD, and alternatively known as a Steiner system S$(2,k,v)$.  

There are necessary divisibility conditions for existence of a PBD$(v,K)$.  These are
\begin{eqnarray}
\label{local-pbd}
v-1 &\equiv& 0 \pmod{\alpha(K)} ~\text{and}\\
\label{global-pbd}
v(v-1)  &\equiv& 0 \pmod{\beta(K)},
\end{eqnarray}
where $\alpha(K):=\gcd\{k-1: k \in K\}$ and $\beta(K):=\gcd\{k(k-1): k \in K\}$.

A set $K$ of integers is {\em PBD-closed} \cite{RMW0} if the existence of a PBD$(v,K)$ implies $v \in K$.  From Wilson's existence theory of designs \cite{RMW1}, every PBD-closed set $K$ (containing integers greater than one) is eventually periodic with period $\beta(K)$.  We can use this to obtain eventual periodicity for the parameter $v$ in our problem, and complete the proof of the main result.

\begin{proof}[Proof of Theorem~\ref{main}]
				Let $S_G=\{v \in \mathbb{Z} : K_{v}^{[\boldsymbol{\mu}(v); \lambda]} \text{ is $G$-decomposable}\}$, where we let 
				$$\boldsymbol{\mu}(v)=\lam (v-1) \Big(\frac{\ell_1}{2m}, \ldots, \frac{\ell_c}{2m}\Big).$$ 
				We first show that this set is PBD-closed, that is, $v\in S_G$ whenever there exists a PBD$(v,S_G)$. If a PBD$(v,S_G)$ exists with blocks $\mathcal{B}=\{B_1, B_2, \ldots, B_t\}$, then $K_v$ can be decomposed into subgraphs, each of which is a $K_{v_j}$ with vertex set $B_j\in \mathcal{B}$ for some $v_j \in S_G$. Similarly, by taking $\lambda$ copies of the subgraphs in the decomposition of $K_v$, we obtain a decomposition of $K_v^{\lambda}$ into the subgraphs $K_{v_j}^{\lambda}$.  Putting these decompositions together provides a decomposition of the ordinary edges of $K_v^{\lambda}$.

It remains to verify the loop conditions.  Since each $v_j\in S_G$, it follows that $K_{v_j}^{[\boldsymbol{\mu}(v_j); \lambda]}$ is $G$-decomposable.  We therefore attach loops with multiplicities $\bm{\mu}(v_j)$ to each vertex of $B_j$. 
For any element $u$, the sum of (ordinary edge) degrees at $u$ within blocks $B_j$ is equal to the degree of $u$ in $K_v^{\lambda}$. That is, we have
				$$
				\sum\limits_{\substack{u \in B_j \\ |B_j|=v_j}} \lambda (v_j-1) = \lambda (v-1).
				$$
				Multiplying each side by $\frac{\ell_i}{2m}$, we obtain
				$$
				\sum\limits_{\substack{u \in B_j \\ |B_j|=v_j}} \bm{\mu}(v_j) = \left( \frac{\lambda \ell_1 (v-1)}{2m},\dots,\frac{\lambda \ell_c (v-1)}{2m} \right)=\bm{\mu}(v).
				$$
Thus we have a $G$-decomposition of $K_v^{[\boldsymbol{\mu}(v); \lambda]}$. So $v\in S_G$ and it follows that $S_G$ is PBD-closed. 
			
				Since, by Proposition \ref{construction for large prime powers}, $S_G$ contains all sufficiently large primes $1 \pmod{2m}$, we know from Dirichlet's theorem that
$$\beta(S_G)=\gcd\{u(u-1) : u\in S_G\}=2M$$
for some positive multiple $M$ of $m$.  Taking this $M$ in Proposition~\ref{wilson's construction}, we see that $S_G$ intersects every admissible residue class modulo $2M$. But $S_G$ is eventually periodic, so our proof is complete.
			\end{proof}

\section{Extensions}
\label{sec:extensions}

\subsection{Digraphs}

\label{sec:digraphs}
Extending the main result to the setting of directed graphs is straightforward, and requires only minor changes.  Let $G$ be a directed graph with $n$ vertices, $m>0$ arcs, and, as before, $l_i$ loops of color $i$.  For a vertex $u \in V(G)$, we let $d_u^-$ and $d_u^+$ denote the in-degree and out-degree, respectively, of $u$ and $e_{u,i}$ the number of loops at $u$ of color $i$.  The divisibility conditions for designs in this setting become
\begin{equation} 
\label{digraph-div} 
\begin{aligned}
m & \mid  \lam v(v-1) ~\text{and}\\ 
\alpha^* & \mid  \lam (v-1),
\end{aligned}
\end{equation} 
where $\alpha^*$ is the least positive integer such that 
			\begin{equation}\label{digraph-alpha}
\alpha^* \left( 1,1, \frac{\ell_1}{m} , \dots, \frac{\ell_c}{m} \right) \in \sum_{u \in V(G)} \left( d_u^-,d_u^+ ,e_{u,1}, \dots, e_{u,c} \right) \Z.
			\end{equation}
The proof of sufficiency of \eqref{digraph-div} for large $v$ can follow Wilson's treatment of directed graphs in \cite{Wilson75}, with a minor adaptation to include vertex loops.  The only place where a notable difference from the undirected case occurs is in the proof of Proposition~\ref{integral solutions}, where \eqref{digraph-alpha} is used.  The finite-field constructions and PBD closure are essentially identical.

We could not envision any unique application modeled by the relaxation to digraphs, although degree-balanced decompositions could now be defined according to total degree or refined according to in-degree/out-degree pairs.  Also, to our knowledge a directed graph version of balanced $G$-decompositions has not been considered in general.

A further extension to the setting of edge-colored directed graphs with colored loops should be routine, at least in the case of decompositions into a single such graph.  We omit the details, which are again repetitious, but encourage the search for decomposition problems that might fit into this more general framework.

\subsection{Families of graphs}
\label{sec:families}

Let $\mathcal{G}$ be a family of graphs, which for the moment we assume have no loops.  A $\mathcal{G}$-\emph{decomposition} of $K_v^\lam$ is a collection of subgraphs, each isomorphic to a graph in $\mathcal{G}$, whose edge sets partition the multiset $E(K_v^\lam)$.  A special case of \cite[Theorem 1.2]{LW} gives an existence theory for $\mathcal{G}$-decompositions of large complete graphs.

\begin{thm}[Lamken and Wilson; see \cite{LW}]
\label{families}
There exists a $\mathcal{G}$-decomposition of $K_v^{\lam}$ for all sufficiently large integers $v$ satisfying
			\begin{equation}\label{family necessary conditions}
			\begin{aligned}
\beta(\mathcal{G}) & \mid  \lam v(v-1) ~\text{and}\\ 
\alpha(\mathcal{G}) & \mid  \lam (v-1), 
			\end{aligned}		
			\end{equation}
where $\beta(\mathcal{G}) = 2 \gcd\{|E(G)|:G \in \mathcal{G}\}$ and $\alpha(\mathcal{G}) = \gcd\{\deg_G(x): x \in V(G), G \in \mathcal{G}\}$.
\end{thm}
An extension of this result to include loops presents some major obstacles.  For starters, the necessary conditions become much more complicated than \eqref{family necessary conditions}, since the loop multiplicities $\mu_i$ of the host graph need not be linked to the proportion of loops in each block.  The first author's thesis \cite[Chapter 5]{Flora thesis} details the arithmetic necessary conditions for $\mathcal{G}$-decompositions in the presence of loops.

To highlight the difficulty obtaining general existence results on graph families with loops -- even of one color -- consider the family $\mathcal{K}=\{K_n^{[1,1]}: n \ge 2\}$ of cliques with exactly one loop at each vertex.  It is trivial to obtain $\mathcal{K}$-decompositions of $K_v^{[\mu;1]}$ for some $\mu$ using only the cliques of size two.  The challenge is to achieve relatively small values of $\mu$, say $\mu=o(v)$.  Indeed, for $\mu$ near $\sqrt{v}$, the $\mathcal{K}$-decomposition problem is very nearly equivalent to deciding existence of a projective plane of order $\mu$.

Even finite families $\mathcal{G}$ pose difficulties (at least for our methods) when certain graphs in the family have loops (of some color) and others do not.  Achieving small loop multiplicities in the host graph would appear to require a relatively small but `balanced' spanning set of blocks such that the leftover graph is admissible for decomposition into the blocks without loops.

There is additional arithmetic complexity in allowing loop colors.  Indeed the the $\mathcal{G}$-decomposition problem in the presence of $c$ loop colors essentially requires integer lattice considerations in dimension $c$.  To illustrate this, consider the family $\mathcal{L}$ consisting of a single loopless edge $K_2$, as well as copies of $K_1$ having loop vectors $\mathbf{L}_j=(\ell_{j,1},\dots,\ell_{j,c})$, $j=1,2,\dots$.  A vector $\bm{\mu}$ of loop multiplicities for the host graph is allowed if and only if it is a nonnegative integer combination of the vectors $\mathbf{L}_j$.

One basic case settled in \cite{Flora thesis} involves the family $\mathcal{K}_k$, each of whose graphs is a clique $K_k$ with one loop at each of $1,2,\dots,k$  vertices.   In this case, a $(v,k,\lam)$-BIBD can be found, and then the loops placed on blocks afterward so as to balance exactly $\mu$ loops at each of the $v$ elements, where $\lam (v-1)/k(k-1) \le \mu \le \lam(v-1)/(k-1)$.  The max-flow/min-cut theorem is invoked to find a distribution of loops onto the blocks.  We do not expect that such methods can always extend to more general families.  However, given a graph family $\mathcal{G}$, loops can be ignored and a $\mathcal{G}$-decomposition of large cliques with possibly unbalanced loops can be obtained using Theorem~\ref{families}.  Arranging these larger tiles so as to balance loops is possibly amenable to network flow methods.

A partial result on families $\mathcal{G}$ can be obtained from Theorem~\ref{main} using a single graph $G_0$ as the disjoint union of (copies of) each $G \in \mathcal{G}$.  However, the divisibility conditions for such $G_0$ are  in general stronger than those for $\mathcal{G}$.  Nevertheless, this idea is used in \cite{LW} to produce constructions of cyclotomic $\mathcal{G}$-decompositions and in \cite{block proportions,graph proportions} to build designs with a prescribed proportion of each size/type of block.

\subsection{Hypergraphs}

Working from the recent existence theory for $t$-designs and hypergraph decompositions, \cite{GKLO,Keevash2},  we can envision the use of non-uniform hypergraphs to enforce various extra balance conditions.  As an example, the problem of achieving desired block proportions in $t$-designs, say locally at $s$-subsets, $s \le t$, is one possible use of such a model. In a different direction,  the large set problem for Steiner triple systems can be modeled analogously as resolvable designs, say using the family of three-vertex edge-colored hypergraphs $G_i$ having three edges $\{1,2\}$, $\{1,3\}$, $\{2,3\}$, each of color $i$, and the common (black) edge $\{1,2,3\}$.  General results of this type may be challenging, however, since the number of colors used in the host graph would depend on its order.


\begin{thebibliography}{99}

			
		\bibitem{hierarchy} A. Bonisoli, S. Bonvicini, and G. Rinaldi, A hierarchy of balanced graph-designs, \textit{Quaderni di Matematica} 28 (2012),  151--163.
		
		
		\bibitem{degree and orbit} S. Bonvicini, Degree- and orbit-balanced $\Gamma$‐designs when $\Gamma$ has five vertices, \textit{J. Combin. Designs} 21 (2013), 359--389.
	
		\bibitem{Flora thesis} F.C.~Bowditch, \textit{Localized Structure in Graph Decompositions}, M.Sc.~Thesis, University of Victoria, 2019.
					
		\bibitem{block proportions} C. J. Colbourn and V. R\"{o}dl, Percentages in pairwise balanced designs, \textit{Discrete Mathematics} 77 (1989),  57--63.

\bibitem{DMW}
A.~Draganova, Y.~Mutoh and R.M.~Wilson, More on decompositions of edge-colored complete graphs, \emph{Discrete Math.} 308 (2008), 2926--2943.

		\bibitem{thickly resolvable} P. Dukes, A. Ling, and A. Malloch, Thickly-resolvable block designs, \textit{Australas. J. Combin.} 64 (2016), 379--391.
		
		\bibitem{Peter and Amanda} P. Dukes and A. Malloch, An existence theory for loopy graph decompositions, \textit{J. Combin. Designs} 19 (2011), 280--289.
				
\bibitem{equitable colorings 2} L. Gionfriddo, M. Gionfriddo and G. Ragusa, Equitable specialized block colorings for 4-cycle systems -- I, \textit{Discrete Math} 310 (2010), 3126--3131.
		
\bibitem{equitable colorings 1} M. Gionfriddo and G. Quattrocchi, coloring 4-cycle systems with equitable colored blocks, \textit{Discrete Math} 284 (2004), 137--148.


\bibitem{GKLO}
S.~Glock, D.~K\"{u}hn, A.~Lo and D.~Osthus,
The existence of designs via iterative absorption, arXiv preprint
\url{http://arxiv.org/abs/1611.06827}, 2016.


		
\bibitem{Hell and Rosa} P. Hell and A. Rosa, Graph decompositions, handcuffed prisoners and balanced $P$-designs, \textit{Discrete Math} 2 (1972), 229--252.

\bibitem{HR}
P. Hor\'{a}k and A. Rosa, Decomposing Steiner triple systems into small configurations, Ars Combin.  26 (1988), 91--105.
	
\bibitem{Keevash2}
P.~Keevash, The existence of designs II, arXiv preprint
\url{http://arxiv.org/abs/1802.05900}, 2018.

\bibitem{LW}
E.R.~Lamken and R.M.~Wilson, Decompositions of edge-colored complete graphs.
{\em J. Combin. Theory Ser. A} 89 (2000), 149--200.


\bibitem{equitable colorings 3} S. Li and C. A. Rodger, Equitable block-colorings of $C_4$-decompositions of $K_v-F$. \textit{Discrete Math} 339 (2016), 1519--1524.
		
\bibitem{Stinson-comm}
D.R.~Stinson, private communication, 2019.

\bibitem{RMW0}
R.M.~Wilson,
An existence theory for pairwise balanced designs I: Composition theorems and morphisms. 
\emph{J. Combin. Theory Ser. A}
13 (1972), 220--245. 

\bibitem{RMW1}
R.M.~Wilson, An existence theory for pairwise balanced designs II: The
structure of PBD-closed sets and the existence conjectures.
\emph{J. Combin. Theory Ser. A}
13 (1972), 246--273.

\bibitem{RMW2}
R.M.~Wilson, An existence theory for pairwise balanced designs III: Proof of
the existence conjectures.  {\em J. Combin. Theory Ser. A}
18 (1975), 71--79.

\bibitem{Wilson75}
R.M.~Wilson, Decompositions of complete graphs into subgraphs isomorphic
to a given graph. {\em Congressus Numerantium} XV (1975), 647--659.


\bibitem{graph proportions} R.M. Wilson, The proportion of various graphs in graph designs, \textit{Combinatorics and Graphs: The Twentieth Anniversary Conference of IPM Combinatorics}, American Mathematical Society, Providence, RI (2010), 251--255.

\end{thebibliography}
\end{document}